\documentclass[a4paper,11pt,reqno]{amsart}

\usepackage{amsmath}
\usepackage[T1]{fontenc}
\usepackage{amssymb}
\usepackage{amsthm}
\usepackage{color}
\usepackage[lmargin=2.5 cm,rmargin=2.5 cm,tmargin=3.5cm,bmargin=2.5cm,paper=a4paper]{geometry}

\newcommand{\Ab}{\mathbf A}

\newcommand{\R}{\mathbb R}

\newcommand{\C}{\mathbb C}

\newcommand{\gse}{\mathrm E_{\rm gs}}
\DeclareMathOperator{\curl}{curl}

\newtheorem{thm}{Theorem}[section]

\newtheorem{lem}[thm]{Lemma}

\theoremstyle{remark}
\newtheorem{rem}[thm]{Remark}

\numberwithin{equation}{section}

\title[Energy of BEC]{Energy of a rotating Bose-Einstein condensate in a harmonic trap}
\author[A. Kachmar]{Ayman Kachmar$^*$}
\thanks{$^*$ Lebanese University, Department of Mathematics, Hadat, Lebanon, and Lebanese International University, School of Arts and Sciences, Beirut, Lebanon.
Email: {\tt ayman.kashmar@liu.edu.lb}}

\date{May 15, 2013}

\begin{document}

\begin{abstract}
The state of a rotating Bose-Einstein condensate in a harmonic trap
is modeled by a wave function  that minimizes the Gross-Pitaevskii
functional. The resulting minimization problem  has two new features
compared to other similar functionals arising in condensed matter
physics, such as the Ginzburg-Landau functional. Namely, the wave
function is defined in all the plane and is  normalized relative to
the $L^2$-norm.
This paper deals with the situation when the coupling constant tends
to $0$ (Thomas-Fermi regime) and the rotation speed is large
compared with the first critical speed.
 It is
given the leading order estimate of the ground state energy together
with the location of the vortices of the minimizing wave function in
the bulk of the condensate. When the rotation speed is inversely
proportional to the coupling constant, the condensate is confined in
an elliptical region  whose conjugate  diameter shrinks and whose
transverse diameter expands  as the rotation speed increases.
\end{abstract}

\maketitle

\section{Introduction}

The analysis of energy functionals modeling rotating Bose-Einstein
condensation is currently an important field of mathematical
physics. A lot of mathematical papers addressed several questions
related to this physical phenomenon. In \cite{LSY, BCP}, it is
proved that the Gross-Pitaevskii frame work is a valid approximation
of the $N$-body model of rotating Bose-Einstein condensation. The
monograph \cite{Af} contains original results as well as many open
questions regarding various models in the subject (see also the
papers \cite{AAB, AD, AJ} and the references therein). A series of
important contributions (\cite{CY, Retal} and references therein)
contain a deep analysis that describes the various critical speeds
of rotating Bose-Einstein condensates in {\it anharmonic} traps.

When the atoms of the condensate are confined in a {\it harmonic}
trap, the Gross-Pitaevskii functional to study is:
\begin{equation}\label{eq:GP}
F_\varepsilon(u)=\int_{\R^2}\left(|(\nabla-i\Omega\Ab_0)u|^2+\frac1{2\varepsilon^2}\Big([a(x)-|u|^2]^2-[a_-(x)]^2\Big)
-\frac{\Omega^2}{4}|x|^2|u|^2\right)\,dx\,.
\end{equation}
The functional in \eqref{eq:GP} is defined for functions satisfying
the {\it mass constraint},
\begin{equation}\label{eq:mc}
\int_{\R^2}|u|^2\,dx=1\,.
\end{equation}
The parameter $\varepsilon>0$ is the coupling constant;
$\varepsilon$ is the ratio of two characteristic lengths. The
parameter $\Omega$ measures the rotational speed,
$\Ab_0(x)=x^\bot/2=(-x_2/2,x_1/2)$, $a(x)=a_0-|x|_\Lambda^2$,
$a_0=\sqrt{2\Lambda/\pi}$, $|x|_\Lambda=\sqrt{x_1^2+\Lambda^2
x_2^2}$\,.

The parameter $\Lambda\in(0,1]$ is fixed as well as the term $a_0$
in the function $a$. The choice of the term $a_0$ forces the
function $a$ to satisfy the normalization condition
$\displaystyle\int_{\R^2}\big(a(x))_-\,dx=1$.

The form of the functional given in \eqref{eq:GP} is adequate to
apply the techniques developed for the Ginzburg-Landau functional.
In non-dimensional units, the functional  that appears in the
physical literature is actually the sum of three terms: the kinetic
energy, the potential energy and the non-linear interaction term
(see e.g. \cite{MCDW}),
\begin{equation}\label{eq:GP-phys}
F_\varepsilon(u)=\int_{\R^2}\left(|\nabla u|^2+\frac1{2\varepsilon^2}\Big([a(x)-|u|^2]^2-[a_-(x)]^2\Big)
-\Omega\, x^\bot\cdot (iu,\nabla u)\right)\,dx\,.
\end{equation}
In the regime $\varepsilon\ll1$ and $\varepsilon\Omega\ll1$, the
condensate is confined in the region
\begin{equation}\label{eq:D}
\mathcal D=\{x\in\R^2~:~a(x)>0\}\,.
\end{equation}

The ground state energy is:
\begin{equation}\label{eq-gse}
\gse(\varepsilon,\Omega)=\inf\,\{\,F_\varepsilon(u)~:~u\in H^1(\R^2)\,,~|x|^2u\in L^2(\R^2)~\&~\int_{\R^2}|u|^2\,dx=1\,\}\,.
\end{equation}
The minimization problem in \eqref{eq-gse} is studied in \cite{IM}
when $\varepsilon\to0_+$ and $\Omega\approx|\ln\varepsilon|$. Among
other things, it is found a critical speed
$\Omega_c=\omega_c|\ln\varepsilon|$ such that minimizers start to
have zeros when $\Omega>\Omega_c$. In this paper, the focus will be
on the regime when $\varepsilon\to0_+$ and $\Omega\gg \Omega_c$.
Part of the results of this paper are qualitatively very similar to
those of \cite{CY, CPRY, CRY} where  {\it flat} and {\it anharmonic}
traps are treated. However,  a regime in the {\it harmonic} trap
discussed in this paper seems to display a new behavior of the
concentration of the condensate's wave function. This is explicitly
discussed in Remark~\ref{rem:concentration} below.

It is established in \cite[Prop.~3.1]{IM} that there is a minimizer
of the problem \eqref{eq-gse} when $\Omega<2\Lambda/\varepsilon$.
The functional in \eqref{eq-gse} is {\it not} bounded from below
when $\Omega>2\Lambda/\varepsilon$.

Setting $\Omega=0$ into the {\it magnetic term } in $F_\varepsilon$,
it is obtained the reduced functional:
\begin{equation}\label{eq:E}
E_{\varepsilon,\Omega}(u)=\int_{\R^2}\left(|\nabla u|^2+\frac1{2\varepsilon^2}\Big([a(x)-|u|^2]^2-[a_-(x)]^2\Big)-\Omega^2\frac{|x|^2}4|u|^2\right)\,dx\,.
\end{equation}
The ground state energy of this functional is:
\begin{equation}\label{eq-gse:E}
e_{\varepsilon,\Omega}=\inf\,\{E_\varepsilon(u)~:~u\in H^1(\R^2)\,,~|x|^2u\in L^2(\R^2)~\&~\int_{\R^2}|u|^2\,dx=1\,\}\,.
\end{equation}
The reduced functional in \eqref{eq:E} is studied in
\cite[Thm.~2.2]{IM} when $\Omega=0$, where  it is established that
\eqref{eq-gse:E} has a positive minimizer
$\widetilde\eta_{\varepsilon}$. In Section~\ref{sec:p}, it will be
constructed a positive minimizer
$\widetilde\eta_{\varepsilon,\Omega}$ of  the functional in
\eqref{eq:E}.  Following an idea of \cite{LM} and writing
$u=\widetilde\eta_{\varepsilon,\Omega} v$, there holds the following
decomposition:

\begin{equation}\label{eq:FG}
F_\varepsilon(u)=E_\varepsilon(\widetilde\eta_\varepsilon)+\mathcal
G_\varepsilon(v)\,,\end{equation} with
\begin{equation}\label{eq:G}
\mathcal G_\varepsilon(v)=\int_{\R^2}\left(\widetilde\eta_{\varepsilon,\Omega}^2|(\nabla-i\Omega\Ab_0)v|^2+\frac{\widetilde\eta_{\varepsilon,\Omega}^4}{2\varepsilon^2}(1-|v|^2)^2
\right)\,dx\,.
\end{equation}
Also, if $u$ is selected as a minimizer of \eqref{eq-gse}, then $v$
will be a minimizer of $\mathcal G_\varepsilon$  under the {\it
weighted} mass constraint,
\begin{equation}\label{eq:mmc}
\int_{\R^2}\widetilde\eta_{\varepsilon,\Omega}^2|v|^2\,dx=1\,.
\end{equation}
More precisely, the minimization problem \eqref{eq-gse} is
equivalent to
\begin{equation}\label{eq-gse:G}
C_0(\varepsilon,\Omega)=\inf\,\{\mathcal G_\varepsilon(v)~:~v\in H^1(\R^2)\,,~\widetilde\eta_{\varepsilon,\Omega}|x|v\in L^2(\R^2)~\&~\int_{\R^2}
\widetilde\eta_{\varepsilon,\Omega}^2|v|^2\,dx=1\,\}\,.
\end{equation}

The main theorem of this paper is:

\begin{thm}\label{thm:K}
Let  $M\in(0,2\Lambda)$ and $b:(0,1)\to (0,\infty)$ satisfies
$\displaystyle\lim_{\varepsilon\to0_+}b(\varepsilon)=\infty$.
Suppose that the rotational speed satisfies:
$$b(\varepsilon)|\ln\varepsilon|\leq\Omega\leq
\frac{M}{\varepsilon}\,,\quad\big(\varepsilon\in(0,1)\big)\,.
$$ There exist a constant $\varepsilon_0>0$ and a function ${\rm err}:(0,\varepsilon_0]\to\R$ such that,
$$\lim_{\varepsilon\to0_+}{\rm err}(\varepsilon)=0\,,$$
and
\begin{equation}\label{eq:K}
\gse=e_{\varepsilon,\Omega}+\Omega\left[\ln\frac1{\varepsilon\sqrt{\Omega}}\right]\Big(1+{\rm err}(\varepsilon)\Big)\,,\quad \big(\varepsilon\in(0,\varepsilon_0)\big)\,.
\end{equation}
Here $\gse$ is introduced in \eqref{eq-gse} and
$e_{\varepsilon,\Omega}$ in \eqref{eq-gse:E}.
\end{thm}

\begin{rem}\label{rem:proof}
In light of the decomposition in \eqref{eq:FG}, the proof of
Theorem~\ref{thm:K} is done by establishing that:
$$C_0(\varepsilon,\Omega)=\Omega\left[\ln\frac1{\varepsilon\sqrt{\Omega}}\right]\Big(1+{\rm
err}(\varepsilon)\Big)\,.$$
\end{rem}

\begin{rem}\label{rem:concentration} ({\bf Bulk of the
condensate})\\
In Section~\ref{sec:p}, it will be shown that the function
$\widetilde\eta_{\varepsilon,\Omega}$ is concentrated in the region
$$
\mathcal D_{\varepsilon\Omega}=\{x\in\R^2~:~\alpha_{\varepsilon\Omega}-|x|^2_{\widetilde
\Lambda_{\varepsilon\Omega}}>0\}\,,
$$
where
$$
\alpha_{\varepsilon\Omega}=a_0\left(\frac{1-\frac{\varepsilon^2\Omega^2}{4\Lambda^2}}{1-\frac{\varepsilon^2\Omega^2}{4}}\right)^{1/4}
\quad{\rm and}\quad
\widetilde\Lambda_{\varepsilon\Omega}=\Lambda\left(\frac{1-\frac{\varepsilon^2\Omega^2}{4\Lambda^2}}{1-\frac{\varepsilon^2\Omega^2}4}\right)^{1/2}\,.$$
It is worthy to discuss the form of the region $\mathcal D_{\varepsilon\Omega}$ in the various existing regimes. In the isotropic case $\Lambda=1$, the region $\mathcal D_{\varepsilon\Omega}$ is independent of $\varepsilon\Omega$,
$$\mathcal D_{\varepsilon\Omega}=\mathcal D=\{x\in\R^2~:~a(x)>0\}\,.$$
In the non-isotropic case, $0<\Lambda<1$, one observes an interesting behavior. If $\varepsilon\Omega\ll1$, then the region $\mathcal
D_{\varepsilon\Omega}$ occupies  $\mathcal D$.

This region shrinks along the $x_1$-axis and expands along the
$x_2$-axis as $\varepsilon\Omega$ increases. If
$\Omega=M/\varepsilon$ and $M\in(0,2\Lambda)$, then as
$M\to2\Lambda$, the region $\mathcal D_{\varepsilon\Omega}$
approaches the following region
$$
\mathcal D_{2\Lambda}=\{0\}\times\R\,.
$$
It seems that this kind of bahavior of the `bulk' of the condensate is new comapred to the existing behavior for anharmonic and flat traps.
\end{rem}

\begin{rem}\label{rem:confinement} ({\bf Concentration of the condensate's wave
function})\\
Let $\delta>0$ and $\mathcal N_{\delta}=\{x\in\mathcal
D_{\varepsilon\Omega}~:~\alpha_{\varepsilon\Omega}-|x|^2_{\widetilde\Lambda_{\varepsilon\Omega}}>\delta\}$.
A simple consequence of the energy asymptotics in
Remark~\ref{rem:proof} and the discussion in
Remark~\ref{rem:concentration} is that any minimizer
$u=\widetilde\eta_{\varepsilon,\Omega}\,v$ of the functional in
\eqref{eq:GP} satisfies,
$$|v|= \left|\frac{u}{\widetilde\eta_{\varepsilon,\Omega}}\right|\to1\quad{\rm in }~L^2\big(\mathcal N_{\delta}\big)\,.$$
Since the functions $u$ and $\widetilde\eta_{\varepsilon,\Omega}$
are normalized in $L^2$, then  the function $u$ satisfies
$$\int_{\mathcal N_\delta}|u|^2\,dx=1+\mathcal O(\delta)\quad{\rm
and}\quad\int_{\R^2\setminus\mathcal N_\delta}|u|^2\,dx=\mathcal O
(\delta)\,,$$ for sufficiently small values of $\delta$. Note that
the behavior of $\widetilde\eta_{\varepsilon\Omega}$ described in
Theorem~\ref{thm:IM0'} is used.
\end{rem}

\begin{rem}
Along the proof of Theorem~\ref{thm:K}, one gets information about
the qualitative behavior of the minimizers. More precisely, it is
possible to get information about  the arrangement of vortices. This
is discussed in Section~\ref{sec:v}.
\end{rem}

\begin{rem}
The letter $C$ denotes a positive constant independent of
$\varepsilon$ and $\Omega$, and whose value is not the same when
seen in different  formulas. The quantity $\mathcal O(B)$ is any
expression that remains in the interval $(-C|B|,C|B|)$. Writing
$A\ll B$ means that $A=\delta B$ and $\delta\to0$. The meaning of
$A\approx B$ is that $A$ is bounded between $c_1B$ and $c_2B$ with
$c_1$ and $c_2$ being positive constants.
\end{rem}

\section{Preliminaries}\label{sec:p}

Some basic properties of the positive minimizer
$\widetilde\eta_{\varepsilon,\Omega}$ of \eqref{eq-gse:E} as well as
of minimizers of the modified problem \eqref{eq-gse:G} will be used
along the  proof of Theorem~\ref{thm:K}. These properties are
recalled here.

\subsection{The unconstrained problem}

The first step is to study the minimization of \eqref{eq:E} without
the mass constraint. The results here are given in \cite{IM} but for
a slightly more particular case on the potential $\widetilde a(x)$
defined below. The proofs here are identically the same as in
\cite{IM} and are not repeated.

Consider the potential
$$\widetilde a(x)=\widetilde a_0-|x|_{\widetilde\Lambda}^2=\widetilde a_0-x_1^2-\widetilde\Lambda^2x_2^2\,,\quad\Big(x=(x_1,x_2)\in\R^2\Big)\,,$$
where $\widetilde a_0$ and $\widetilde\Lambda$ are positive
parameters. The parameters $\widetilde a_0$ and $\widetilde \Lambda$
may depend on $\varepsilon$ and $\Omega$ but they should remain
bounded between two positive constants $c_1$ and $c_2$ that are
independent of $\varepsilon$ and $\Omega$. The results in this
section are valid under this last assumption.

Consider the functional
\begin{equation}\label{eq:E'}
\widetilde E_{\varepsilon}(u)=\int_{\R^2}\left(|\nabla u|^2+\frac1{2\varepsilon^2}\Big([\widetilde a(x)-|u|^2]^2-[\widetilde a_-(x)]^2\Big)\right)\,dx\,.
\end{equation}
The functional in \eqref{eq:E'} will be minimized over
configurations in the space
$$\mathcal H=\{u\in H^1(\R^2)~:~|x|^2u\in L^2(\R^2)\}\,.$$
The proof of Theorem~\ref{thm:IM0} below is given  in
\cite[Proposition~2.1]{IM}.


\begin{thm}\label{thm:IM0}
There exist two positive constants $\varepsilon_0>0$ and $C>0$ such
that, if $\varepsilon\in(0,\varepsilon_0)$, then there is a
real-valued minimizer $\eta_\varepsilon=\eta_{\varepsilon,\widetilde
a}\in \mathcal H$ of \eqref{eq:E'}   satisfying:
\begin{enumerate}
\item $E_\varepsilon(\eta_\varepsilon)\leq
C|\ln\varepsilon|$ and $\eta_\varepsilon>0$ in $\R^2$\,;
\item $\eta_\varepsilon$ is the unique solution of
$$-\Delta\eta_\varepsilon=\frac1{\varepsilon^2}\big(\widetilde a-
\eta_\varepsilon^2)\eta_\varepsilon\quad{\rm and}\quad \eta_\varepsilon>0\quad {\rm in~}\R^2\,.$$
\item $\eta_\varepsilon(x)\leq
C\varepsilon^{1/3}\exp\big(\widetilde
a(x)/(4\varepsilon^{2/3})\big)$ if $|x|_{\widetilde
\Lambda}\geq\sqrt{\widetilde a_0^{\,}}$\,;
\item $(1-C\varepsilon^{1/3})\sqrt{\widetilde a(x)}\leq \eta_\varepsilon(x)\leq
\sqrt{\widetilde a(x)}$ if $|x|_{\widetilde\Lambda}\leq
\sqrt{\widetilde a_0^{\,}}-\varepsilon^{1/3}$\,.
\item $\eta_\varepsilon(x)\leq C\varepsilon^{1/3}$ if $
\sqrt{\widetilde a_0^{\,}}-\varepsilon^{1/3}\leq|x|_{\widetilde
\Lambda}\leq\sqrt{\widetilde a_0^{\,}}$\,.
\end{enumerate}
%
\end{thm}


\subsection{The constrained problem}

This section is devoted to the construction of a positive minimizer
of the constrained problem in \eqref{eq-gse:E}.

A standard compactness argument shows the existence of a minimizer
$u_{\varepsilon,\Omega}$ of \eqref{eq-gse:E}. The details are given
in \cite{IM}. Since $\big|\nabla|u_{\varepsilon,\Omega}|\,\big|\leq
|\nabla u_{\varepsilon,\Omega}|$, then $|u_{\varepsilon,\Omega}|$ is
a minimizer of \eqref{eq-gse:E} too. This discussion leads to the
existence of a positive minimizer
$\widetilde\eta_{\varepsilon,\Omega}=|u_{\varepsilon,\Omega}|$ of
\eqref{eq-gse:E}. The Euler-Lagrange equation satisfied by
$\widetilde\eta_{\varepsilon,\Omega}$ is,
$$-\Delta\widetilde\eta_{\varepsilon,\Omega}=
\frac1{\varepsilon^2}\big(k_\varepsilon\varepsilon^2+V_{\varepsilon\Omega}-\widetilde\eta_{\varepsilon,\Omega}^2\big)\widetilde\eta_{\varepsilon,\Omega}\,,$$
where $k_\varepsilon\in\R$ is the Lagrange multiplier and
$V_{\varepsilon\Omega}(x)=a_0-|x|_\Lambda^2+\frac{\varepsilon^2\Omega^2}4|x|^2$.

Multiplying both sides of the Euler-Lagrange equation by
$\widetilde\eta_{\varepsilon,\Omega}$, integrating by parts and
using
$\displaystyle\int_{\R^2}\widetilde\eta_{\varepsilon,\Omega}^2\,dx=1$
yield that $a_0+k_\varepsilon\varepsilon^2>\mu_\varepsilon>0$, where
$\mu_\varepsilon$ is the first eigenvalue of the Schr\"odinger
operator
$$-\Delta+\frac1{\varepsilon^2}\left(|x|^2_{\widetilde\Lambda}-\frac{\varepsilon^2\Omega^2}{4}|x|^2\right)\quad{\rm in~}L^2(\R^2)\,.$$
Note that, by the assumption on $\Omega$ and $\Lambda$, the
potential of the operator is positive and goes to $\infty$ when
$|x|\to\infty$.

Define
$$\widetilde\varepsilon=\left(1-\frac{\varepsilon^2\Omega^2}4\right)^{-1/2}\frac{a_0}{a_0+k_\varepsilon\varepsilon^2}\,\varepsilon\,,\quad
\nu_\varepsilon(x)=\sqrt{\frac{a_0}{a_0+k_\varepsilon\varepsilon^2}}\,\widetilde\eta_{\varepsilon,\Omega}
\left(\sqrt{\frac{a_0+k_\varepsilon\varepsilon^2}{a_0}}\,x\right)\,.$$
The function $\nu_\varepsilon$ satisfies,
$$-\Delta\nu_\varepsilon=\frac1{\widetilde\varepsilon^2}\big(\widetilde
a-\nu_\varepsilon^2\big)\nu_\varepsilon\,,\quad
\nu_\varepsilon>0\quad {\rm in~ \R^2}\,,$$ where
$$\widetilde a(x)=\widetilde a_{\varepsilon\Omega}=\widetilde a_0-|x|_{\widetilde\Lambda}^2\,,\quad
\widetilde a_0=\frac{a_0}{1-\frac{\varepsilon^2\Omega^2}4}\,,\quad \widetilde\Lambda^2=\frac{\Lambda^2-\frac{\varepsilon^2\Omega^2}4}{1-\frac{\varepsilon^2\Omega^2}4}\,.$$
The conclusion (2) in Theorem~\ref{thm:IM0} asserts that,
$$\nu_\varepsilon(x)=\eta_{\widetilde\varepsilon,\widetilde
a}(x)\quad (x\in\R^2)\,,$$ where $\eta_{\widetilde\varepsilon,\widetilde a}$ is the solution of the unconstrained problem.
As a consequence, there holds,
$$\widetilde\eta_{\varepsilon,\Omega}(x)=\sqrt{\frac{a_0+k_\varepsilon\varepsilon^2}{a_0}}\,\eta_{\widetilde\varepsilon,\widetilde a}
\left(\sqrt{\frac{a_0}{a_0+k_\varepsilon\varepsilon^2}}\,x\right)\,.$$

Thanks to the conclusions (3)-(5) in Theorem~\ref{thm:IM0} and the
mass constraint
$\displaystyle\int_{\R^2}\widetilde\eta_{\varepsilon,\Omega}^2\,dx=1$,
there holds,
\begin{align*}
\left(\frac{a_0}{a_0+k_\varepsilon\varepsilon^2}\right)^2&=\left(\int_{\widetilde a(x)>0}\widetilde a(x)\,dx\right)\big(1+\mathcal O(\varepsilon^{1/3})\big)\\
&=\Lambda\left(\Lambda^2-\frac{\varepsilon^2\Omega^2}{4}\right)^{-1/2}\left(1-\frac{\varepsilon^2\Omega^2}{4}\right)^{-3/2}\,
\big(1+\mathcal O(\varepsilon^{1/3})\big)\,.
\end{align*}
In the sequel, let,
\begin{multline}\label{eq:pe}
\alpha_{\varepsilon\Omega}
=a_0\left(\frac{1-\frac{\varepsilon^2\Omega^2}{4\Lambda^2}}{1-\frac{\varepsilon^2\Omega^2}{4}}\right)^{1/4}\,,\quad
\widetilde\Lambda_{\varepsilon\Omega}=\Lambda\left(\frac{1-\frac{\varepsilon^2\Omega^2}{4\Lambda^2}}{1-\frac{\varepsilon^2\Omega^2}4}\right)^{1/2}\\
p_{\varepsilon\Omega}(x)=\Big(
\alpha_{\varepsilon\Omega}-\,|x|^2_{\widetilde
\Lambda_{\varepsilon\Omega}}\Big)=\sqrt{\frac{a_0+k_\varepsilon\varepsilon^2}{a_0}}\,\widetilde
a\left(\sqrt{\frac{a_0}{a_0+k_\varepsilon\varepsilon^2}}\,x\right)\big(1+\mathcal
O(\varepsilon^{1/3})\big)\,.
\end{multline}
 Now, an immediate application of
Theorem~\ref{thm:IM0} leads to:

\begin{thm}\label{thm:IM0'}
Let $M\in(0,2\Lambda)$.  There exist positive constants
$\varepsilon_0$, $C$ and $\delta_0$ such that, if
$\varepsilon\in(0,\varepsilon_0)$ and $\Omega\in[0,M)$, then there
is a real-valued minimizer $\widetilde\eta_{\varepsilon,\Omega}$ of
the constrained problem \eqref{eq:E'}   satisfying:
\begin{enumerate}
\item $E_\varepsilon(\widetilde\eta_{\varepsilon,\Omega})\leq
C\Omega^2$ and $\widetilde\eta_{\varepsilon,\Omega}>0$ in
$\R^2$\,;
\item $\widetilde\eta_{\varepsilon,\Omega}(x)\leq
C\varepsilon^{1/3}\exp\big(\delta_0
p_{\varepsilon\Omega}(x)/(\varepsilon^{2/3})\big)$ if
$p_{\varepsilon\Omega}(x)\leq -\delta_0\varepsilon^{1/3}$\,;
\item $(1-C\varepsilon^{1/3})\sqrt{p_{\varepsilon\Omega}(x)}\leq \widetilde\eta_{\varepsilon,\Omega}(x)\leq
\sqrt{p_{\varepsilon\Omega}(x)}$ if
$p_{\varepsilon\Omega}(x)\geq \delta_0\varepsilon^{1/3}$\,;
\item $\eta_\varepsilon(x)\leq C\varepsilon^{1/3}$ if $-\delta_0\varepsilon^{1/3}\leq p_{\varepsilon\Omega}(x)\leq \delta_0\varepsilon^{1/3}$\,.
\end{enumerate}
\end{thm}

\subsection{A uniform bound of the ground states}

\begin{thm}\label{thm:IM'}
Let $M\in(0,2\Lambda)$.  There exist positive constants $C$,
$\delta$, $\lambda$ and $\varepsilon_0$ such that, if
$\varepsilon\in(0,\varepsilon_0)$ and $0<\Omega\leq M/\varepsilon$,
then every minimizer $v_\varepsilon$ of \eqref{eq-gse:G} satisfies:
$$|\widetilde\eta_\varepsilon v_\varepsilon(x)|\leq C\left(\sqrt{\frac1{2\Lambda-M}}\,+1\right)\quad{\rm
in}~\R^2\,.$$
\end{thm}
\begin{proof}
Under the assumption on the rotational speed, Proposition~3.2 in
\cite{IM} implies that the problem \eqref{eq-gse} has a minimizer
$u_\varepsilon$. In light of the decomposition in \eqref{eq:FG}, it
follows that
$v_\varepsilon=u_\varepsilon/\widetilde\eta_\varepsilon$ is a
minimizer of the problem \eqref{eq-gse:G}. Theorem~\ref{thm:IM'}
will be proved by establishing properties of $u_\varepsilon$. The
function $u_\varepsilon$ satisfies
\begin{equation}\label{eq-ue}
-(\nabla-i\Omega\Ab_0) u_\varepsilon=\frac1{\varepsilon^2}\big(a(x)+\frac14\varepsilon^2\Omega^2|x|^2+\varepsilon^2\ell_\varepsilon-|u_\varepsilon|^2\big)u_\varepsilon\quad{\rm
in}~\R^2\,,\end{equation} where $\ell_\varepsilon\in\R$ is the lagrange multiplier. Furthermore, it holds (see the derivation of \cite[(3.7)\&(3.11)]{IM}):
\begin{equation}\label{eq-en:ue}
F_\varepsilon(u_\varepsilon)\leq C\Omega^2\,,\quad |\ell_\varepsilon|\leq C\varepsilon^{-1}\Omega\,,\quad
\int_{\R^2\setminus\mathcal D}|u_\varepsilon|^4\,dx\leq C\varepsilon^2\Omega^2\,.
\end{equation}
Let $U_\varepsilon=|u_\varepsilon|^2$ and
$b(x)=a(x)+\frac14\varepsilon^2\Omega^2|x|^2+\varepsilon^2\ell_\varepsilon$.
In light of the identity,
$${\rm
Re}\Big[\overline{u_\varepsilon}\,(\nabla-i\Omega\Ab_0)^2u_\varepsilon\Big]=\frac12\Delta
U_\varepsilon-|(\nabla-i\Omega\Ab_0)u_\varepsilon|^2\,,$$  the
function $U_\varepsilon$ satisfies,
\begin{equation}\label{eq:Ue}
\frac12\Delta U_\varepsilon\geq
-\frac1{\varepsilon^2}(b(x)-U_\varepsilon)U_\varepsilon\quad{\rm
in~}\R^2\,.
\end{equation}
Let $\lambda>\sqrt{a_0}$\,, $\mathbb E=\{x\in\R^2~:~|x|\geq
2\lambda\,\}$ and $\Theta=\{x\in\R^2~:~|x|> \lambda\}$. The
condition on $\lambda$ ensures that
$\Theta\subset\R^2\setminus\mathcal D$. In the set $\Theta$, there
holds,
$$b(x)\leq a_0-\lambda^2(\Lambda^2-M^2)+\varepsilon^2\ell_\varepsilon\leq -\lambda^2\left(\Lambda^2-\frac{M^2}4\right)+C\,.$$
As a consequence, it is possible to select the constant $\lambda\geq
\sqrt{\frac{2C}{\Lambda^2-\frac{M^2}2}}$ such that the function
$U_\varepsilon$ is subharmonic in the open set $\Theta$.

Consider an arbitrary point $x_0\in\mathbb E$. The definition of the
set $\Theta$ yields that $B(x_0,\lambda)\subset\Theta$ and
$\Theta\subset\R^2\setminus \mathcal D$. Since the function
$U_\varepsilon$ is subharmonic and its $L^2$-norm is estimated in
\eqref{eq-en:ue}, then there exists a constant $C_*>0$ such that,
$$0\leq U_\varepsilon(x_0)\leq
\frac1{|B(x_0,\lambda)|}\int_{B(x_0,\lambda)}U_\varepsilon^2(x)\,dx\leq
\mathcal O\left(\frac1\lambda\varepsilon\Omega\right)\leq \frac{C_*}{\lambda}\,.$$
The next step is to prove that $U_\varepsilon$ is bounded in the set
$$B_{r}=\{x\in\R^2\setminus\mathcal
D~:~|x|\leq r\}
$$
where $r=3\lambda$\,. Select a positive constant $C$ such that
$b(x)\leq C\lambda+\frac{C_*}{\lambda}$ in $B_{r}$.  Notice that
$\partial B_{r}\subset\mathbb E$ and consequently,
$U_\varepsilon\leq C_*\leq C\lambda+\frac{C_*}{\lambda}$ in
$\partial B_{r}$. Thus, if the maximum of $U_\varepsilon$ in $B_{r}$
is greater than $C\lambda+\frac{C_*}{\lambda}$, then  the point of
maximum  is an interior point in $B_{r}$. It is impossible that such
a point of maximum exists. In fact, if there exists a point of
maximum $x_0$ satisfying
$C\lambda+\frac{C_*}{\lambda}-U_\varepsilon(x_0)<0$, then $\Delta
U_\varepsilon(x_0)\leq 0$. This leads to a contradiction in light of
the following inequality,
$$\frac12\Delta
U_\varepsilon+\frac1{\varepsilon^2}\left(C\lambda+\frac{C_*}{\lambda}-U_\varepsilon\right)U_\varepsilon\geq
0\,,$$ which  results from \eqref{eq:Ue} and the choice of the
constant $C$.
\end{proof}

\begin{rem}\label{rem-ve}
There is a simple  consequence of Theorem~\ref{thm:IM'} and (3) in
Theorem~\ref{thm:IM0'}. Let $K$ be a compact set and $\delta>0$. If
$K\subset\{x\in\R^2~:~p_{\varepsilon\Omega}(x)>\delta\}$ for
sufficiently small values of $\varepsilon$, then there exist
constants $\varepsilon_{K,\delta}$ and $C_{K,\delta}$ such that, for
all $\varepsilon\in(0,\varepsilon_{K,\delta})$,
$|v_\varepsilon(x)|\leq C_{K,\delta}$ in $K$.

Here, the function $p_{\varepsilon\Omega}(x)$ is introduced in
\eqref{eq:pe}.
\end{rem}

\section{Reduced Ginzburg-Landau energy}\label{sec:gl}

Let $K=(-1/2,1/2)\times(-1/2,1/2)$ be a square of unit side length,
$\lambda$, $h_{\rm ex}$ and $\varepsilon$ be  positive parameters.
Consider the functional defined for all $u\in H^1(K;\C)$,
\begin{equation}\label{eq-2D-f}
E^{\rm 2D}_\lambda(u)=\int_K\left(|(\nabla-ih_{\rm ex}\Ab_0)u|^2+\frac\lambda{2\varepsilon^2}(1-|u|^2)^2\right)\,dx\,.
\end{equation}
Here $\Ab_0$ is the vector potential whose $\curl$ is equal to $1$,
\begin{equation}\label{eq-Ab0}
\Ab_0(x_1,x_2)=\frac12(-x_2,x_1)\,,\quad
(x_1,x_2)\in\R^2\,.\end{equation}
Notice that the functional $E^{\rm 2D}_\lambda$ is a simplified
version of the full Ginzburg-Landau functional considered in
\cite{SS00}, as the magnetic potential in \eqref{eq-2D-f} is given
and {\it not} an unknown of the problem.

Minimization of the functional $E^{\rm 2D}_\lambda$ arises naturally
over `magnetic periodic' functions. Let us introduce the following
space,
\begin{multline}\label{space-p}
E_{h_{\rm ex}}=\{u\in H^1_{\rm
loc}(\R^2;\C)~:~u(x_1+1,x_2)=e^{ih_{\rm
ex}x_2/2}u(x_1,x_2)\,,\\u(x_1,x_2+1)=e^{-ih_{\rm
ex}x_1/2}u(x_1,x_2)\}\,,
\end{multline}
together with the ground state energy,
\begin{equation}\label{eq-mp}
m_{\rm p}(h_{\rm ex},\varepsilon)=\inf\{E^{\rm 2D}_\lambda(u)~:~u\in E_{h_{\rm ex}}\}\,.
\end{equation}
Minimization of $E^{\rm 2D}_\lambda$ over configurations without
prescribed boundary conditions will be needed as well. The ground
state energy of this problem is,
\begin{equation}\label{eq-m0}
m_{0}(h_{\rm ex},\varepsilon)=\inf\{E^{\rm 2D}_\lambda(u)~:~u\in H^1(K)\}\,.
\end{equation}
 The ground state energies $m_0(h_{\rm ex},\varepsilon)$ and  $m_{\rm
p}(h_{\rm ex},\varepsilon)$ are estimated in \cite{K-3D} by
borrowing tools from \cite{SS} and \cite{SS00}. This is recalled in
the next theorem.

\begin{thm}\label{thm-SS-p}
Assume that $\lambda_2>\lambda_1>0$ are given constants,
$\lambda\in(\lambda_1,\lambda_2)$ and $h_{\rm ex}$ is a function of
$\varepsilon$ such that
$$|\ln\varepsilon|\ll h_{\rm ex}\ll \frac1{\varepsilon^2}\,,\quad{\rm as}~\varepsilon\to0\,.$$
As $\varepsilon\to 0$, the ground state energies $m_{0}(h_{\rm
ex},\varepsilon)$ and $m_{\rm p}(h_{\rm ex},\varepsilon)$ satisfy,
$$m_{0}(h_{\rm ex},\varepsilon)=h_{\rm ex}\ln\frac1{\varepsilon\sqrt{h_{\rm ex}}}\big(1+o(1)\big)\quad{\rm and}\quad
m_{\rm p}(h_{\rm ex},\varepsilon)=h_{\rm ex}\ln\frac1{\varepsilon\sqrt{h_{\rm
ex}}}\big(1+o(1)\big)\,.
$$
Here, the expression $o(1)$ tends to $0$ as $\varepsilon\to0$
uniformly with respect to $\lambda$.
\end{thm}

In the forthcoming section, it will be needed a trial state
satisfying the mass constraint ($L^2$-norm equal to $1$) and
having an energy close to $m_{\rm p}(h_{\rm ex},\varepsilon)$. The
next Lemma provides one with  a useful trial state whose
$L^2$-norm is {\it close} to $1$.

\begin{lem}\label{lem:test-fuc}
Suppose that $\lambda>0$, $h_{\rm ex}$ and $\varepsilon$ are as in
Theorem~\ref{thm-SS-p}. There exists a function $f_\varepsilon$ in
$H^1(K)$ such that
$$|f_\varepsilon|\leq 1 \quad{\rm in}~K\,,$$
$$\{x\in K~:~|f_\varepsilon(x)|<1\}\subset \bigcup_{i=1}^{\mathsf n} B(a_i,\varepsilon)\quad {\rm and}\quad \mathsf n=\mathcal O(h_{\rm ex})\,,$$
$$1-\mathcal O(\varepsilon^2h_{\rm ex})\leq \int_{K}|f_\varepsilon(x)|^2\,dx\leq 1\,,$$
and
$$E^{\rm 2D}_\lambda(f_\varepsilon)\leq h_{\rm ex}\ln\frac1{\varepsilon\sqrt{h_{\rm
ex}}}\big(1+o(1)\big)\,,$$ as $\varepsilon\to0_+$. Furthermore,
$f_\varepsilon$ is independent of $\lambda$, and $\mathcal O$ is
uniform with respect to $\lambda$.
\end{lem}
\begin{proof}
For the convenience of the reader, the construction of
$f_\varepsilon$ is outlined.  Details can be found in \cite{AyS}.
Let $N$ be the largest positive integer satisfying $N\leq
\sqrt{h_{\rm ex}/2\pi}<N+1$. Divide the square $K$ into $N^2$
disjoint squares $(K_j)_{0\leq j\leq N^2-1}$ each of side length
equal to $1/N$ and center $a_j$.  Let $h$ be the unique solution
of the problem,
$$\left\{
\begin{array}{lll}
-\Delta h+h_{\rm ex}=2\pi\delta_{a_0}&{\rm in}&K_0\\
\displaystyle\frac{\partial h}{\partial\nu}=0&{\rm on}&\partial K_0\\
\displaystyle\int_{K_0}h\,dx=0.&&
\end{array}\right.
$$
Here $\nu$ is the unit outward normal vector of $K_0$. The function
$h$  satisfies periodic conditions on the boundary of $K_0$, and
\begin{align*}
\int_{K_0\setminus B(a_0,\varepsilon)}|\nabla h|^2\,dx\leq 2\pi \ln\frac1{\varepsilon N}+\mathcal O(1)= 2\pi\ln\frac1{\varepsilon\sqrt{h_{\rm ex}}}+\mathcal O(1)
\,,\quad{\rm as}~\varepsilon\to0_+\,.
\end{align*}
The function $h$ is extended  by periodicity in the square $K$.
Let $\phi$ be a function (defined modulo $2\pi$) satisfying in
$K\setminus\{a_{j}~:~0\leq j\leq N^2-1\}$,
$$\nabla\phi=-\nabla^\bot h+h_{\rm ex}\Ab_0\,,\quad \big(\nabla^\bot=(-\partial_{x_2},\partial_{x_1})\big).$$
Here  $\Ab_0$ is the magnetic potential in \eqref{eq-Ab0}. If $x\in
K_0$, let $\rho(x)=\min(1,|x-a_0|/\varepsilon)$. The function $\rho$
is extended  by periodicity in the square $K$. Put
$f_\varepsilon(x)=\rho(x) e^{i\phi(x)}$ for all $x\in K$. The
function $f_\varepsilon$ can be extended as  a function in the space
$E_{h_{\rm ex}}$ in \eqref{space-p}, see \cite[Lemma~5.11]{Ay} for
details.

 The energy of $f_\varepsilon$ is easily computed, since
$f_\varepsilon$ is `magnetic periodic' and $N=\sqrt{h_{\rm
ex}/2\pi}\big(1+o(1)\big)$.  Clearly, in the square $K_0$,
$|f_\varepsilon(x)|<1$ if and only if $|x-a_0|<\varepsilon$. Thus,
it is easy to check that $f_\varepsilon$ satisfies the requirements
in Lemma~\ref{lem:test-fuc}.
\end{proof}

\section{Upper Bound}\label{sec:ub}
\subsection{The test configuration}
Recall the definition of the ground state energy
$C_0(\varepsilon,\Omega)$ in \eqref{eq-gse:G}. The assumption on the
rotational speed $\Omega$ is $|\ln\varepsilon|\ll\Omega\leq
M/\varepsilon$ with $M\in(0,2\Lambda)$. Let
$$L>\sqrt{a_0\left(1-\frac{M^2}{4}\right)^{-1/4}}\quad {\rm and}\quad 0<\delta<\min\left(\sqrt{ a_0\left(1-\frac{M^2}{4\Lambda^2}\right)}\,,\,\frac{L}{2}\right)\,.$$
Recall the definition of $\alpha_{\varepsilon\Omega}$ in
\eqref{eq:pe}. The constants $\delta$ and $L$ are selected so that
$$\delta<\sqrt{\alpha_{\varepsilon\Omega}}\,<L\quad{\rm and}\quad
\sqrt{\alpha_{\varepsilon\Omega}}\,+\delta<L\,.$$
Define,
$$\mathcal U_L=\{x\in\mathcal D~:~|x|_{\widetilde\Lambda_{\varepsilon\Omega}}<L\}\,.$$
Thanks to the assumption on $\Omega$, if $\varepsilon$ is sufficiently small, then there holds the inclusion,
$$  \mathcal D_{\varepsilon\Omega}=\{x\in\R^2~:~p_{\varepsilon\Omega}(x)>0\}\subset \mathcal U_L\,,$$
where $\widetilde\Lambda_{\varepsilon\Omega}$ and
$p_{\varepsilon\Omega}$ are introduced in \eqref{eq:pe} and
$\displaystyle\int_{p_{\varepsilon\Omega}(x)>0}p_{\varepsilon\Omega}(x)\,dx=1$.

Define
\begin{equation}\label{eq:lh}
\ell=\left(\frac\Omega{|\ln\varepsilon|}\right)^{1/4}\frac1{\sqrt{\Omega}}\,,\quad h_{\rm ex}=\frac1{\ell^2}\,.\end{equation}
Recall the ground state energy $m_{\rm p}(h_{\rm ex},\varepsilon)$
and the space $E_{h_{\rm ex}}$
 introduced in  \eqref{eq-mp} and \eqref{space-p} respectively. Let $f_\varepsilon\in E_{h_{\rm ex}}$ be the test function defined in Lemma~\ref{lem:test-fuc}.
 In particular, $f_\varepsilon$ satisfies $E^{\rm 2D}_\lambda(f_\varepsilon)\leq  h_{\rm ex}\ln\frac1{\varepsilon\sqrt{h_{\rm
ex}}}\big(1+o(1)\big)$ for any $\lambda$ varying between two
positive constants $\lambda_1$ and $\lambda_2$.

Define,
$$v(x)=\chi(x)\,f_\varepsilon\Big(\ell\sqrt{\Omega}\,x\Big)\quad(x\in\R^2)\,,$$
where $\chi$ is a cut-off function satisfying,
$$0\leq \chi\leq 1{\rm ~in~}\R^2\,,\quad \chi(x)=0{\rm ~when~}|x|_{\widetilde\Lambda_{\varepsilon\Omega}}\geq 2L\,,\quad \chi(x)=1~{\rm when~}
|x|_{\widetilde\Lambda_{\varepsilon\Omega}}\leq L\,,$$ and
$$|\nabla\chi|\leq \frac{C}L\quad{\rm in~}\R^2\,.$$

 Let $(\mathcal K_j)$ be the lattice of $\R^2$ generated by the cube,
$$\mathcal K=\bigg(-\frac1{2\ell\sqrt{\Omega}},\frac1{2\ell\sqrt{\Omega}}\bigg)\times \bigg(-\frac1{2\ell\sqrt{\Omega}},\frac1{2\ell\sqrt{\Omega}}\bigg)\,.$$
Let $\mathcal J=\{\mathcal K_j~:~\mathcal K_j\cap\mathcal
U_{2L}\not=\emptyset\}$ and $N={\rm Card}\,\mathcal J$. As
$\varepsilon\to0_+$, the number $N$ satisfies,
$$N=|\mathcal U_{2L}|\times \big(\ell\sqrt{\Omega}\big)^2\,\big(1+o(1)\big)\,.$$
In light of Lemma~\ref{lem:test-fuc} and the exponential decay of
$\widetilde\eta_{\varepsilon,\Omega}$ in Lemma~\ref{thm:IM0},  the
function $v$ satisfies,
\begin{equation}\label{eq-nv}
1-\mathcal
O(\varepsilon^2\Omega)\leq\int_{\R^2}\widetilde\eta_{\varepsilon,\Omega}^2|v|^2\,dx\leq
1\,.\end{equation} Define the test function,
\begin{equation}\label{eq-testfunc}\widetilde
v(x)=\frac{v(x)}{\sqrt{\displaystyle\int_{\R^2}\widetilde\eta_\varepsilon^2|v|^2\,dx}}\,.\end{equation}
Clearly, the function $\widetilde v$ satisfies the weighted mass
constraint,
\begin{equation}\label{eq-wmc}
\int_{\R^2}\widetilde\eta_\varepsilon^2|\widetilde
v|^2\,dx=1\,,\end{equation} and consequently, there holds the upper
bound $C_0(\varepsilon, \Omega)\leq \mathcal
G_\varepsilon(\widetilde v)$.  The rest of the section will be
devoted to estimating the energy $\mathcal G_\varepsilon(\widetilde
v)$. It will be established that:
\begin{equation}\label{eq-ub*}
\limsup_{\varepsilon\to
0_+}\left(\frac{\mathcal G_\varepsilon(\widetilde v)}{2\Omega\left[\ln\frac1{\varepsilon\sqrt{\Omega}}\right]}-1\right)\leq
0\,.\end{equation} The next estimate \eqref{eq-ub} is a consequence of \eqref{eq-ub*},
\begin{equation}\label{eq-ub}
C_0(\varepsilon,\Omega)\leq  \Omega\left[\ln\frac1{\varepsilon\sqrt{\Omega}}\right]\big(1+{\rm err}(\varepsilon)\big)\,.
\end{equation}

\subsection{Energy of the test configuration: Proof of \eqref{eq-ub*}}
It will be shown that the term
$$C_\varepsilon=\mathcal G_{\varepsilon}(\widetilde v)=\displaystyle\int_{\R^2}\left(\widetilde\eta_\varepsilon^2|(\nabla-i\Omega\Ab_0)\widetilde
v|^2+\frac{\widetilde\eta_\varepsilon^4}{2\varepsilon^2}(1-|\widetilde
v|^2)^2\right)\,dx$$ is of leading order equal to
$L_\varepsilon=\Omega\left[\ln\displaystyle\frac1{\varepsilon\sqrt{\Omega}}\right]\,$.
It is useful to write $C_\varepsilon$ as the sum of four terms,
\begin{equation}\label{eq-C-ep}
C_\varepsilon=C_{\varepsilon,1}+C_{\varepsilon,2}+C_{\varepsilon,3}+C_{\varepsilon,4}\,,
\end{equation}
where
\begin{align}
&C_{\varepsilon,1}=\int_{|x|_{\widetilde\Lambda_{\varepsilon\Omega}}\leq \sqrt{\alpha_{\varepsilon\Omega}}\,-\delta}
\left(\widetilde\eta_\varepsilon^2|(\nabla-i\Omega\Ab_0)\widetilde
v|^2+\frac{\widetilde\eta_\varepsilon^4}{2\varepsilon^2}(1-|\widetilde
v|^2)^2\right)\,dx\,,\label{eq-C-ep1}\\
&C_{\varepsilon,2}=\int_{\sqrt{\alpha_{\varepsilon\Omega}}\,-\delta\leq |x|_{\widetilde\Lambda_{\varepsilon\Omega}}\leq \sqrt{\alpha_{\varepsilon\Omega}}\,+\delta}\left(\widetilde\eta_\varepsilon^2|(\nabla-i\Omega\Ab_0)\widetilde
v|^2+\frac{\widetilde\eta_\varepsilon^4}{2\varepsilon^2}(1-|\widetilde
v|^2)^2\right)\,dx\,,\label{eq-C-ep2}\\
&C_{\varepsilon,3}=\int_{\sqrt{\alpha_{\varepsilon\Omega}}\,+\delta\leq|x|_{\widetilde\Lambda_{\varepsilon\Omega}}\leq 2L}\left(\widetilde\eta_\varepsilon^2|(\nabla-i\Omega\Ab_0)\widetilde
v|^2+\frac{\widetilde\eta_\varepsilon^4}{2\varepsilon^2}(1-|\widetilde
v|^2)^2\right)\,dx\,,\label{eq-C-ep3}\\
&C_{\varepsilon,4}=\int_{|x|_{\widetilde\Lambda_{\varepsilon\Omega}}\geq 2L}\left(\widetilde\eta_\varepsilon^2|(\nabla-i\Omega\Ab_0)\widetilde
v|^2+\frac{\widetilde\eta_\varepsilon^4}{2\varepsilon^2}(1-|\widetilde
v|^2)^2\right)\,dx\,,\label{eq-C-ep4}
%
\end{align}
and $\alpha_{\varepsilon\Omega}$ is as in \eqref{eq:pe}.

\subsubsection*{The term $C_{\varepsilon,1}$:}
Let $\mathcal J_0=\{j\in\mathcal J~:~\mathcal K_j\cap
\{x~:~|x|_{\widetilde\Lambda_{\varepsilon\Omega}}\leq
\sqrt{\alpha_{\varepsilon\Omega}}\,-\delta\}\not=\emptyset\}$. Since
$\delta$ is selected independently of $\varepsilon$, then in light
of Theorem~\ref{thm:IM0'}, there holds in every square $\mathcal
K_j$ with $j\in\mathcal J_0$,
$$\widetilde \eta_\varepsilon^2(x)\leq
p_{\varepsilon\Omega}(x)\,.
$$
The mean value theorem applied to the function
$p_{\varepsilon\Omega}$ yields,
$$
p_{\varepsilon\Omega}(x)\leq
p_{\varepsilon\Omega}(x_j)+\frac{C}{\ell\sqrt{\Omega}}\,,$$ where
$x_j$ is an arbitrary point  in $\mathcal K_j$ and $j\in\mathcal
J_0$. The above two estimates applied successively yield an upper
bound of the term $C_{\varepsilon,1}$ as follows:
$$C_{\varepsilon,1}\leq\sum_{j\in\mathcal J_0}
\Big[p_{\varepsilon\Omega}(x_j)+\frac{C}{\ell\sqrt{\Omega}}\,\Big]\int_{\mathcal
K_j}\left(|(\nabla-i\Omega\Ab_0)\widetilde
v|^2+\frac{\lambda_\varepsilon}{2\varepsilon^2}(1-|\widetilde
v|^2)^2\right)\,dx\,,$$ where
$$\lambda_\varepsilon=\max_{j\in\mathcal J_0}\left(\frac{p_{\varepsilon\Omega}(x_j)}{p_{\varepsilon\Omega}(x_j)+\frac{C}{\ell\sqrt{\Omega}}}\right)\,.$$

In the domain $\mathcal U_L$, the function $\chi$ is equal to $1$
and $v(x)=f_\varepsilon(\ell\sqrt{\Omega}\,x)$. By using
successively the estimate in \eqref{eq-nv}, the `magnetic'
periodicity of $v$ over the lattice $(\mathcal K_j)_j$ and the bound
$|v|\leq 1$, one gets the following upper bound,
\begin{align}
\int_{\mathcal
K_j}\Big(|(\nabla-i\Omega\Ab_0)&\widetilde
v|^2+\frac{\lambda_\varepsilon}{2\varepsilon^2}(1-|\widetilde
v|^2)^2\Big)\,dx\nonumber\\
&\leq(1+C\varepsilon^2\Omega)\int_{\mathcal
K_j}\left(|(\nabla-i\Omega\Ab_0)
v|^2+\frac{\lambda_\varepsilon}{2\varepsilon^2}(1-|
v|^2)^2\right)\,dx+C\Omega\int_{\mathcal K_j}|v|^4\,dx\nonumber\\
&\leq (1+C\varepsilon^2\Omega)\int_{\mathcal
K}\left(|(\nabla-i\Omega\Ab_0)
v|^2+\frac{\lambda_\varepsilon}{2\varepsilon^2}(1-|
v|^2)^2\right)\,dx+C\Omega|\mathcal K_j|\,.\label{eq-v}
\end{align}
The integral term in \eqref{eq-v} is computed by the   change of
variable $y=\ell\sqrt{\Omega}\,x$ that transforms it to
\begin{equation}\label{eq-v*}
\int_{K}\left(|(\nabla-ih_{\rm ex}\Ab_0)f_\varepsilon|^2+\frac{\lambda_\varepsilon}{2\tilde\varepsilon^2}(1-|f_\varepsilon|^2)^2\right)\,dx\,,
\end{equation}
where $\tilde\varepsilon=\varepsilon\ell\sqrt{\Omega}$ and $ h_{\rm
ex}=\frac1{\ell^2}\,$. As $\varepsilon\to0_+$,
$\tilde\varepsilon\gg\varepsilon$ and $h_{\rm ex}$ satisfies
$|\ln\varepsilon|\ll h_{\rm ex}\ll{\varepsilon}^{-2}$. Also,
$\lambda_\varepsilon$ remains inside a fixed interval
$[\lambda_1,\lambda_2]$. Consequently, it is possible to use
Lemma~\ref{lem:test-fuc} and get that $\big(1+o(1)\big)h_{\rm
ex}\ln\displaystyle\frac1{\varepsilon\sqrt{h_{ex}}}$ is an upper
bound of the  term in \eqref{eq-v*}. As a consequence, it is
obtained the following upper bound of $C_{\varepsilon,1}$,
\begin{equation}\label{eq-C-ep-1}
C_{\varepsilon,1}\leq (1+C\varepsilon^2\Omega) \sum_{j\in\mathcal
J_0}
\Big[p_{\varepsilon\Omega}(x_j)\,+C\varepsilon^2|\ln\varepsilon|+\frac{C}{\ell\sqrt{\Omega}}\,\Big]
\left(\big(1+o(1)\big)h_{\rm ex}\ln\frac1{\varepsilon\sqrt{h_{\rm
ex}}}+C\Omega|\mathcal K_j|\right)\,.\end{equation} Recall that, as
$\varepsilon\to0_+$, the number of squares $\mathcal K_j$ satisfies
$N=|\mathcal U_{2L}|\times\ell^2\Omega\big(1+o(1)\big)$. Since
$|\mathcal K_j|=\displaystyle\frac1{\ell^2\Omega}$ for every $j$,
then $\displaystyle\sum_{j\in\mathcal J}|\mathcal K_j|=|\mathcal
U_{2L}|\big(1+o(1)\big)\,$. Also, all the extra terms appearing in
\eqref{eq-C-ep-1} are $o(1)$ as $\varepsilon\to0_+$, and this leads
one to,
\begin{align*}
C_{\varepsilon,1}&\leq
\big(1+o(1)\big)\sum_{j\in\mathcal J_0} \frac1{|\mathcal K_j|}p_{\varepsilon\Omega}(x_j)\ell^2\Omega h_{\rm ex}\ln\frac1{\varepsilon\sqrt{h_{\rm ex}}}\\
&=\big(1+o(1)\big)\Omega\ln\frac1{\varepsilon\sqrt{\Omega}}\sum_{j\in\mathcal J_0} \frac1{|\mathcal K_j|}a(x_j)\,.
\end{align*}
Since each point $x_j$ is arbitrarily selected in the square
$\mathcal K_j$, then the sum $\sum_j \frac1{|\mathcal
K_j|}p_{\varepsilon\Omega}(x_j)$ becomes a Riemann sum. Select the
points $(x_j)$ such that the sum is a lower Riemann sum. That way,
$$\sum_{j\in\mathcal J'} \frac1{|\mathcal K_j|}p_{\varepsilon\Omega}(x_j)\leq
\int_{|x|_{\widetilde\Lambda_{\varepsilon\Omega}}\leq \sqrt{\alpha_{\varepsilon\Omega}}\,-\delta}p_{\varepsilon\Omega}(x)\,dx
\leq\int_{p_{\varepsilon\Omega}(x)>0}p_{\varepsilon\Omega}(x)\,dx=1 \,.$$ As a consequence, the term $C_{\varepsilon,1}$
satisfies,
\begin{equation}\label{eq-C-ep-1'}
C_{\varepsilon,1}\leq \big(1+o(1)\big)\Omega\ln\frac1{\varepsilon\sqrt{\Omega}}\quad{\rm as~}\varepsilon\to0_+\,.
\end{equation}

\subsubsection*{The term $C_{\varepsilon,2}$:}
To estimate the term $C_{\varepsilon,2}$, it is used the result of
Theorem~\ref{thm:IM0'} that the function
$\widetilde\eta_\varepsilon$ is bounded independently of
$\varepsilon$ to get that,
$$C_{\varepsilon,2}\leq C\int_{\sqrt{\alpha_{\varepsilon\Omega}}\,-\delta\leq|x|_{\widetilde\Lambda_{\varepsilon\Omega}}\leq \sqrt{\alpha_{\varepsilon\Omega}}\,+\delta}\left(|(\nabla-i\Omega\Ab_0)\widetilde
v|^2+\frac1{2\varepsilon^2}(1-|\widetilde v|^2)^2\right)\,dx\,.$$
The definition of $\widetilde v$ and the estimate in \eqref{eq-nv}
together yield,
\begin{multline*}C_{\varepsilon,2}\leq
C(1+C\varepsilon^2\Omega)\int_{\sqrt{\alpha_{\varepsilon\Omega}}\,-\delta\leq|x|_{\widetilde\Lambda_{\varepsilon\Omega}}\leq
\sqrt{\alpha_{\varepsilon\Omega}}\,+\delta}\left(|(\nabla-i\Omega\Ab_0)
v|^2+\frac1{2\varepsilon^2}(1-| v|^2)^2\right)\,dx
 \\+C\Omega\int_{\sqrt{\alpha_{\varepsilon\Omega}}\,-\delta\leq|x|_{\widetilde\Lambda_{\varepsilon\Omega}}\leq \sqrt{\alpha_{\varepsilon\Omega}}\,+\delta}|v|^2\,dx\,.
\end{multline*}
The function $\chi$ is equal to $1$ in
$\{{\sqrt{\alpha_{\varepsilon\Omega}}\,-\delta\leq|x|_{\widetilde\Lambda_{\varepsilon\Omega}}\leq
\sqrt{\alpha_{\varepsilon\Omega}}\,+\delta}\}\subset\mathcal U_L$.
As a consequence $v(x)=f_\varepsilon(\ell\sqrt{\Omega}\,x)$. As is
done for the term $C_{\varepsilon,1}$, one gets that,
\begin{equation}\label{eq-C-ep-2}
C_{\varepsilon,2}\leq C\big(1+o(1)\big)
\left(\int_{\sqrt{\alpha_{\varepsilon\Omega}}\,-\delta\leq|x|_{\widetilde\Lambda_{\varepsilon\Omega}}\leq \sqrt{\alpha_{\varepsilon\Omega}}\,+\delta}\,dx\right)\,\Omega\ln\frac{1}{\varepsilon\sqrt{\Omega}}
\leq C\delta\Omega\ln\frac{1}{\varepsilon\sqrt{\Omega}}\,.\end{equation}

\subsubsection*{The term $C_{\varepsilon,3}$:}
When
${\sqrt{\alpha_{\varepsilon\Omega}}\,+\delta\leq|x|_{\widetilde\Lambda_{\varepsilon\Omega}}\leq2L}$,
the function $\chi$ is no more constant and the function $v$ is {\it
small}. As a consequence, it is not useful to estimate the
`Ginzburg-Landau' energy of $v$ along the  same procedure as done
before. However, as Theorem~\ref{thm:IM0} states, the function
$\widetilde\eta_\varepsilon$ decays exponentially, and this will be
the key to estimate the term $C_{\varepsilon,3}$. Thanks to
\eqref{eq-nv}, the function $\widetilde v$ satisfies the uniform
inequality $|1-|\widetilde v|^2|\leq1+\mathcal
O(\varepsilon^2\Omega)$. This and the exponential decay of
$\widetilde\eta_\varepsilon$ in Theorem~\ref{thm:IM0} together yield
when $\varepsilon\to0_+$,
$$\frac1{2\varepsilon^2}\int_{\sqrt{\alpha_{\varepsilon\Omega}}\,+\delta\leq|x|_{\widetilde\Lambda_{\varepsilon\Omega}}\leq2L}\widetilde\eta_\varepsilon^4(1-|\widetilde
v|^2)^2\,dx\leq
C\frac1{\varepsilon^2}\exp\left(-\frac\delta{\varepsilon^{1/2}}\right)\int_{\sqrt{a_0}+1/2\leq|x|_\Lambda\leq
\sqrt{a_0}+1}\,dx =o(1)\,.$$ Using a similar
reasoning, the kinetic energy term is estimated as follows,
\begin{align*}
\int_{\sqrt{\alpha_{\varepsilon\Omega}}\,+\delta\leq|x|_{\widetilde\Lambda_{\varepsilon\Omega}}\leq2L}&\widetilde\eta_\varepsilon^2|(\nabla-i\Omega\Ab_0)\widetilde
v|^2\,dx\\
 &\leq
C\exp\left(-\frac\delta{\varepsilon^{1/2}}\right)\int_{\sqrt{\alpha_{\varepsilon\Omega}}\,+\delta\leq|x|_{\widetilde\Lambda_{\varepsilon\Omega}}\leq2L}\Big(|(\nabla-i\Omega\Ab_0) v|^2+|\nabla\chi|^2|v|^2\Big)\,dx\\
&\leq C\exp\left(-\frac\delta{\varepsilon^{1/2}}\right)\Omega\ln\frac{1}{\varepsilon\sqrt{\Omega}}=o(1)\,,
\end{align*}
thereby obtaining that $C_{\varepsilon,3}=o(1)$ as
$\varepsilon\to0_+$.

\subsubsection*{The term $C_{\varepsilon,4}$:}
Recall the definition of this term in \eqref{eq-C-ep4} and that  the
function $\widetilde v=0$ here. As a consequence,
$C_{\varepsilon,4}= \displaystyle\int_{|x|_\Lambda\geq
\sqrt{a_0}+1}\frac{\widetilde\eta_\varepsilon^4}{2\varepsilon^2}\,dx$
and this is equal to $o(1)$ as $\varepsilon\to0_+$ after  using the
exponential decay of $\widetilde\eta_\varepsilon$ stated in
Theorem~\ref{thm:IM0'}.

\subsection*{Conclusion:} Collecting the estimates
$C_{\varepsilon,4}=o(1)$, $C_{\varepsilon,3}=o(1)$,
\eqref{eq-C-ep-2} and \eqref{eq-C-ep-1} and inserting them into
\eqref{eq-C-ep} yields an upper bound of $C_\varepsilon$. Inserting
this bound  into the expression of $\mathcal
G_\varepsilon(\widetilde v)$ yields the upper bound
$$C_0(\varepsilon,\Omega)\leq\left(1+C\delta+o(1)\right)\Omega\ln\frac1{\varepsilon\sqrt{\Omega}}
+o(1)\,,$$ as $\varepsilon\to0_+$. This yields \eqref{eq-ub} by taking the successive limits as ${\varepsilon\to0_+}$ and then as ${\delta\to0_+}$.
\section{Lower Bound}\label{sec:lb}

Suppose that $v$ is a minimizer of the functional $\mathcal
G_\varepsilon$ introduced in \eqref{eq:G}, and that the rotational
speed $\Omega$ satisfies the assumption of Theorem~\ref{thm:K}. The
aim of this section is to write a lower bound of $\mathcal
G_\varepsilon(v)$.

The assumption on the rotational speed is still
$|\ln\varepsilon|\ll\Omega\leq M/\varepsilon$ with $0<M<2\Lambda$.
Consider a positive constant
$$0<\delta<\sqrt{a_0\left(1-\frac{M^2}{4\Lambda^2}\right)^{1/4}}$$
and the following subset of $\mathcal D_{\varepsilon\Omega}$,
$$\mathcal U_{\delta}=\{x\in\R^2~:~|x|_{\widetilde\Lambda_{\varepsilon\Omega}}\leq \sqrt{\alpha_{\varepsilon\Omega}}\,-\delta\}\,,$$
where $\alpha_{\varepsilon\Omega}$ and
$\widetilde\Lambda_{\varepsilon\Omega}$ are introduced in
\eqref{eq:pe}.

Recall the lattice of squares $\mathcal K_j$ introduced in
Section~\ref{sec:ub}. The parameters $\ell$ and $h_{\rm ex}$ are
still as in \eqref{eq:lh}. Put
\begin{equation}\label{J'}
\mathcal J'=\{j~:~\mathcal K_j\subset\mathcal U_\delta\}\,.
\end{equation}
There holds the obvious lower bound,
\begin{equation}\label{eq-lb-gl}
\begin{aligned}
&\int_{\R^2}\left(\widetilde\eta_\varepsilon^2|(\nabla-i\Omega\Ab_0)v|^2
+\frac{\widetilde\eta_\varepsilon^4}{\varepsilon^2}(1-|v|^2)^2\right)\,dx\\
&\qquad\qquad\geq
 \int_{\mathcal U_\delta}\left(\widetilde\eta_\varepsilon^2|(\nabla-i\Omega\Ab_0)v|^2
+\frac{\widetilde\eta_\varepsilon^4}{2\varepsilon^2}(1-|v|^2)^2\right)\,dx\\
&\qquad\qquad\geq\sum_{j\in\mathcal J'}\int_{\mathcal K_j}\left(\widetilde\eta_\varepsilon^2|(\nabla-i\Omega\Ab_0)v|^2
+\frac{\widetilde\eta_\varepsilon^4}{2\varepsilon^2}(1-|v|^2)^2\right)\,dx\,.
\end{aligned}
\end{equation}
\subsection*{Lower bound of the `Ginzburg-Landau' energy:}
For each $j\in\mathcal J'$, it will be obtained a lower  bound  of
the term,
\begin{equation}\label{eq-lb-func}
\mathcal G_\varepsilon(v,\mathcal K_j)=\int_{\mathcal K_j}\left(\widetilde\eta_\varepsilon^2|(\nabla-i\Omega\Ab_0)v|^2
+\frac{\widetilde\eta_\varepsilon^2}{2\varepsilon^2}(1-|v|^2)^2\right)\,dx\,.
\end{equation}
By  Theorem~\ref{thm:IM0'}, one can write for an arbitrary point
$x_j$ in $\mathcal K_j$,
$$\widetilde\eta_{\varepsilon}^2(x)\geq
(1-C\varepsilon^{1/3})p_{\varepsilon\Omega}(x)\geq \left(1-C\varepsilon^{1/3}-\frac{C}{\ell\sqrt{\Omega}}\right)p_{\varepsilon\Omega}(x_j)\quad{\rm in~}\mathcal K_j\,,$$
and consequently,
\begin{equation}\label{eq-lb-gl-j}
\mathcal G_\varepsilon(v,\mathcal K_j)\geq \left(1-C\varepsilon^{1/3}-\frac{C}{\ell\sqrt{\Omega}}\right)
\int_{\mathcal K_j}\left(p_{\varepsilon\Omega}(x_j)|(\nabla-i\Omega\Ab_0)v|^2
+\frac{p_{\varepsilon\Omega}(x_j)^2}{2\varepsilon^2}(1-|v|^2)^2\right)\,dx\,.
\end{equation}
Let $y_j$ be the center of the square $\mathcal K_j$,
$K=(-1/2,1/2)^2$, $\tilde\varepsilon=\ell\sqrt{\Omega}\,\varepsilon$
and $h_{\rm ex}=1/\ell^2$. Using the re-scaled function
$f(x)=v(y_j+\ell\sqrt{\Omega}\,x)$\,, $(x\in K)\,$, it is possible
to express   \eqref{eq-lb-gl-j} in the following form,
\begin{equation}\label{eq-lb-gl-j'}
\mathcal G_\varepsilon(v,\mathcal K_j)\geq \left(1-C\varepsilon^{1/3}-\frac{C}{\ell\sqrt{\Omega}}\right)p_{\varepsilon\Omega}(x_j)
\int_{K}\left(|(\nabla-ih_{\rm ex}\Ab_0)f|^2+\frac{p_{\varepsilon\Omega}(x_j)}{2\tilde\varepsilon^2}(1-|f|^2)^2\right)\,dx\,.
\end{equation}
Notice that the term $p_{\varepsilon\Omega}(x_j)$ remains in a
constant interval $[\lambda_1,\lambda_2]$ as $j\in\mathcal J'$ and
$\varepsilon$ vary. Also, as $\varepsilon\to0$, $\tilde\varepsilon$
and $h_{\rm ex}$ satisfy $|\ln\tilde\varepsilon|\ll h_{\rm ex}\ll
\tilde\varepsilon^{-2}$. Thus, it is possible to bound the integral
on the right side of \eqref{eq-lb-gl-j'} by the ground state energy
$m_0(h_{\rm ex},\tilde\varepsilon)$ in \eqref{eq-m0}, which is
estimated from below in Theorem~\ref{thm-SS-p}. Therefore, it is
inferred from \eqref{eq-lb-gl-j'},
\begin{equation}\label{eq-lb-gl-j''}
\mathcal G_\varepsilon(v,\mathcal K_j)\geq \big(1+o(1)\big)
p_{\varepsilon\Omega}(x_j)h_{\rm ex}\ln\frac1{\tilde\varepsilon\sqrt{h_{\rm ex}}}=\big(1+o(1)\big)
p_{\varepsilon\Omega}(x_j)\frac1{\ell^2}\ln\frac1{\varepsilon\sqrt{\Omega}}\,.
\end{equation}
Inserting this into \eqref{eq-lb-func} and then into
\eqref{eq-lb-gl} yields,
\begin{equation}\label{eq-lb-gl'}
\int_{\R^2}\left(\widetilde\eta_\varepsilon^2|(\nabla-i\Omega\Ab_0)v|^2
+\frac{\widetilde\eta_\varepsilon^2}{2\varepsilon^2}(1-|v|)^2\right)\,dx\geq
\big(1+o(1)\big)\Omega\ln\frac1{\varepsilon\sqrt{\Omega}}\sum_{j\in\mathcal J'}\frac1{\ell^2\Omega}\,p_{\varepsilon\Omega}(x_j).
\end{equation}
The sum on the right side of \eqref{eq-lb-gl'} is estimated as
follows. As $\varepsilon\to0_+$, the term
$\displaystyle\sum_{j\in\mathcal J'}\frac1{\ell^2\Omega}\,a(x_j)$ is
a Riemann sum. Select the points $(x_j)$ such that the sum is an
upper Riemann sum. As a consequence, there holds,
\begin{align*}
\sum_{j\in\mathcal J'}p_{\varepsilon\Omega}(x_j)h_{\rm
ex}\ln\frac1{\tilde\varepsilon\sqrt{h_{\rm ex}}}&=\Omega\ln\frac1{\varepsilon\sqrt{\Omega}}\sum_{j\in\mathcal J'}\frac1{\ell^2\Omega}\,p_{\varepsilon\Omega}(x_j)\\
&=\Omega\ln\frac1{\varepsilon\sqrt{\Omega}}\int_{\mathcal
U_{2\delta}}p_{\varepsilon\Omega}(x)\,dx\,.
\end{align*}
Therefore, it results from \eqref{eq-lb-gl'},
\begin{equation}\label{eq-lb-gl-conc}
\int_{\R^2}\left(\widetilde\eta_\varepsilon^2|(\nabla-i\Omega\Ab_0)v|^2
+\frac{\widetilde\eta_\varepsilon^2}{2\varepsilon^2}(1-|v|)^2\right)\,dx\geq
\Omega\ln\frac1{\varepsilon\sqrt{\Omega}}\left(\int_{\mathcal U_{2\delta}}p_{\varepsilon\Omega}(x)\,dx\right)\,.
\end{equation}
Recall that the function $p_{\varepsilon\Omega}$ in \eqref{eq:pe}
satisfies
$\displaystyle\int_{p_{\varepsilon\Omega}(x)>0}p_{\varepsilon\Omega}(x)\,dx=1$.
Thus,
$$\int_{\mathcal
U_{2\delta}}p_{\varepsilon\Omega}(x)\,dx=\int_{p_{\varepsilon\Omega}(x)>0}p_{\varepsilon\Omega}(x)\,dx-\int_{p_{\varepsilon\Omega}(x)>2\delta}p_{\varepsilon\Omega}(x)\,dx
\geq 1-C\delta\,.$$ That way, \eqref{eq-lb-gl-conc} becomes,
\begin{equation}\label{eq-lb-gl-conc'}
\int_{\R^2}\left(\widetilde\eta_\varepsilon^2|(\nabla-i\Omega\Ab_0)v|^2
+\frac{\widetilde\eta_\varepsilon^2}{2\varepsilon^2}(1-|v|)^2\right)\,dx\geq
\Omega\ln\frac1{\varepsilon\sqrt{\Omega}}\left(1-C\delta\right)\,.
\end{equation}

\subsection*{Conclusion:}
It is obtained by collecting the estimate in 
\eqref{eq-lb-gl-conc'},
$$C_0(\varepsilon,\Omega)\geq\Omega\ln\frac1{\varepsilon\sqrt{\Omega}}\left(1-C\delta\right)\,.$$
As a consequence, it is obtained by taking the limit as
$\varepsilon\to0_+$,
$$\liminf_{\varepsilon\to0_+}\frac{C_0(\varepsilon,\Omega)}{\Omega\ln\frac1{\varepsilon\sqrt{\Omega}}}\geq
1-C\delta\,.$$ By Taking $\delta\to 0_+$, it results the lower bound:
$$\liminf_{\varepsilon\to0_+}\frac{C_0(\varepsilon,\Omega)}{\Omega\ln\frac1{\varepsilon\sqrt{\Omega}}}\geq
1\,.$$
The conclusion of this section and Section~\ref{sec:ub} finishes the
proof of Theorem~\ref{thm:K}.

\begin{rem}\label{eq-rgl}
If $U\subset \mathcal D_{\varepsilon\Omega}$  and $u\in H^1(U)$,
define the local energy:
$$\mathcal
E_\varepsilon(u;U)=\int_{U}\left(\widetilde\eta_\varepsilon^2|(\nabla-i\Ab_0)u|^2+\frac{\widetilde\eta_\varepsilon^4}{2\varepsilon^2}(1-|u|^2)^2\right)\,dx\,.$$
The analysis of this section allows one to prove the following. If
$v$ is a minimizer of \eqref{eq-gse:G}, $U\subset\mathcal
D_{\varepsilon\Omega}$ is an open set, $\overline{U}\subset\mathcal
D_{\varepsilon\Omega}$, $|\partial U|=0$, and $U$ is independent of
$\varepsilon$ and $\Omega$, then,
$$\mathcal E_\varepsilon(v;U)\geq
\Omega\ln\frac1{\varepsilon\sqrt{\Omega}}\left(\int_Up_{\varepsilon\Omega}(x)\,dx+o(1)\right)\quad{\rm
as ~}\varepsilon\to0_+\,.$$ Combine this lower bound with  the upper
bound \eqref{eq-ub} to obtain the `local' energy asymptotics:
$$\mathcal E_\varepsilon(v;U)=
\Omega\ln\frac1{\varepsilon\sqrt{\Omega}}\left(\int_Up_{\varepsilon\Omega}(x)\,dx+o(1)\right)\quad{\rm as ~}\varepsilon\to0_+\,.$$
\end{rem}

\section{Vortices and their density}\label{sec:v}
The assumption on the rotational speed is as in Theorem~\ref{thm:K}.
Recall the definition of the domain $\mathcal D$ in \eqref{eq:D}.
Let $\beta>0$. Suppose that $U$ is an open set in $\R^2$ satisfying
the properties in Remark~\ref{eq-rgl} and
$${\rm dist}(U,\partial\mathcal D_{\varepsilon\Omega})\geq\beta\,.$$
According to Theorem~\ref{thm:IM0'}, the function
$\widetilde\eta_\varepsilon$ satisfies the pointwise bound
$\widetilde\eta_\varepsilon\geq c_0(U)>0$ in $U$. The constant
$c_0(U)$ depends only on $U$.

Let $v$ be a minimizer of \eqref{eq-gse:G}. By borrowing the results
of \cite{SS, SS00}, it will be given some details regarding the
location and `density' of the zeros of the minimizer $v$ inside $U$.

Consider the lattice of squares $(\mathcal K_j)$ generated by the
square $\mathcal K=(-\delta,\delta)\times(-\delta,\delta)$, where
$\delta=\frac12\big(|\ln\varepsilon|/\Omega\big)^{-1/4}$. Suppose
that  $x_j$ is the center of the square $\mathcal K_j$.

By Theorem~\ref{thm-SS-p}, there exists a positive function
$g(\varepsilon)$ such that, as $\varepsilon\to0_+$,
$g(\varepsilon)\ll1$ and
$${\rm GL}_\varepsilon(v;\mathcal K_j):=
\int_{\mathcal
K_j}\left(|(\nabla-i\Omega\Ab_0)v|^2+\frac{\widetilde\eta_\varepsilon^2(x_j)}{2\varepsilon^2}(1-|v|^2)^2\right)\,dx\geq
\big(1-g(\varepsilon)\big)\Omega\delta^2\ln\frac1{\varepsilon\sqrt{\Omega}}\,.$$
  One distinguishes between {\it good}
squares and {\it bad} squares in $U$; {\it good} squares are those
satisfying that
$${\rm GL}_\varepsilon(v;\mathcal K_j):=
\int_{\mathcal
K_j}\left(|(\nabla-i\Omega\Ab_0)v|^2+\frac{\widetilde\eta_\varepsilon^2(x_j)}{2\varepsilon^2}(1-|v|^2)^2\right)\,dx\leq
\big(1+\sqrt{g(\varepsilon)}\,\big)\Omega\delta^2\ln\frac1{\varepsilon\sqrt{\Omega}}\,,$$
while {\it bad} squares satisfy the reverse condition that ${\rm
GL}_\varepsilon(v;\mathcal K_j)>
\Omega\delta^2\big(1+\sqrt{g(\varepsilon)}\,\big)\ln\frac1{\varepsilon\sqrt{\Omega}}$.
The number of {\it bad} squares $N_{\rm b}$ is small compared to the
number of {\it good} squares $N_{\rm g}$, namely $N_{\rm b}\ll
N_{\rm g}$ as $\varepsilon\to0_+$. Proposition~5.1 in \cite{SS00}
gives one  the following. There exists a constant $C>0$ and a
positive function $\hat g(\varepsilon)$ such that, if $\mathcal K_j$
is a {\it good} square
then there exists a finite family of discs
$\big(B(a_{i,j},r_{i,j})\big)_i$ with the following properties,
\begin{enumerate}
\item $\displaystyle\sum_i r_{i,j}\leq  C\Omega^{-1/2}$\,;
\item $\{x\in \mathcal K_j~:~|v(x)|<\frac12\}\subset
\displaystyle\bigcup_{i} B(a_{i,j},r_{i,j})$\,;
\item If $d_{i,j}$ is the winding number of $v/|v|$ when
$B(a_{i,j},r_{i,j})\subset \mathcal K_j$ and $0$ otherwise,
then,
$$
\sum_{i}d_{i,j}\geq
\Omega\delta^2\big(1-\hat g(\varepsilon)\big)\quad {\rm and}\quad\sum_{i}|d_{i,j}|\leq
\Omega\delta^2\big(1+\hat g(\varepsilon)\big)\,.$$
\item $\hat g(\varepsilon)\ll 1$ as $\varepsilon\to0_+$.
\end{enumerate}
Let $\mathcal J_{\rm g}$ be the collection of all indices $j$ such
that $\mathcal K_j$ is a good square and $\mathcal K_j\subset U$.
Define the measure \begin{equation}\label{eq-vm} \mathcal
\mu_\varepsilon=\sum_{\substack{i,j\\ j\in\mathcal J_{\rm g}}}
\,d_{i,j}\delta_{a_{i,j}}\,,\end{equation} where $\delta_{a_{i,j}}$
is the dirac measure supported at $a_i$. The measure
$\mu_\varepsilon$ is called the vorticity measure in $U$: It
indicates the existence of vortices (when $\mu_\varepsilon\not=0$),
its support indicates the location of vortices, and its norm
indicates their density.

Notice that the aforementioned construction indicates the location
and density of vortices for  minimizers of \eqref{eq-gse}, since
$v=u/\widetilde\eta_\varepsilon$ and $u$ is a minimizer of
\eqref{eq-gse}. Thus, $v$ and $u$ have the same  zeros (vortices).

It is possible to prove that:
\begin{thm}\label{thm:v}
Under the assumption of Theorem~\ref{thm:K}, the vorticity measure
in $U$ fulfills the weak convergence:
$$\frac1{\Omega}\,\mu_\varepsilon\rightharpoonup \mathbf 1_{U}\,dx\quad{\rm
as}~\varepsilon\to0_+,$$ where $dx$ is the Lebesuge measure in $\R^2$ and $\mathbf 1_{U}$ the characteristic function of $U$.
\end{thm}
\begin{proof}
Notice that the upper bound in (3) and the fact that the number of
indices $j$ is asymptotically proportional to $\delta^{-2}$ together
yield that $\Omega^{-1}\sum_{i,j}|d_{i,j}|$ is bounded independently
of $\varepsilon$ and $\Omega$. Consequently, by passing to a
subsequence, one can suppose that $\Omega^{-1}\mu_\varepsilon$
converges weakly to a measure $\mu$. It suffices to prove that
$\mu=\mathbf 1_{U}\,dx$.

 Since the number of good
squares satisfies $N_{\rm g}\times \delta^2=|U|+o(1)$ as
$\varepsilon\to0_+$, then the  two-sided estimate of $\sum_{i,j}
d_{i,j}$ in (3) above leads  to the following. If $S$ is an open set
in $U$ and $|\partial S|=0$, then
\begin{multline*}\Omega|S|\big(1+o(1)\big)\leq
\sum_{i,j}d_{i,j}\leq \big(1+o(1)\big)\mu_\varepsilon(S)\\\leq
\big(1+o(1)\big)\sum_{i,j}|d_{i,j}|\leq\Omega|S|\big(1+o(1)\big)\,,\quad{\rm
as~}\varepsilon\to0_+\,.\end{multline*} This proves that
$\Omega^{-1}\mu_\varepsilon$ converges weakly to the Lebesgue
measure restricted to $U$.
\end{proof}

\section*{Acknowledgements} The research of the author is partially
supported by a grant of Lebanese University.

\end{document}